\documentclass[preprint,1p,times]{elsarticle}

\usepackage[utf8]{inputenc}

\usepackage[english]{babel}

\usepackage{amsmath,amsfonts,amssymb,amsthm}

\theoremstyle{definition}
\newtheorem{theorem}{Theorem}
\newtheorem{lemma}[theorem]{Lemma}
\newtheorem{proposition}[theorem]{Proposition}
\newtheorem{corollary}[theorem]{Corollary}

\theoremstyle{definition}
\newtheorem{definition}[theorem]{Definition}

\newtheorem{problem}{Problem}

\theoremstyle{remark}
\newtheorem{remark}[theorem]{Remark}

\newcommand{\N}{\mathbb{N}} 
\newcommand{\Z}{\mathbb{Z}} 
\newcommand{\R}{\mathbb{R}} 
\newcommand{\C}{\mathbb{C}} 
\newcommand{\D}{\mathbb{D}} 
\newcommand{\T}{\mathbb{T}} 
\newcommand{\K}{\mathbb{K}} 

\newcommand{\eps}{\varepsilon}

\newcommand{\abs}[1]{\left\lvert#1\right\rvert}

\newcommand{\norm}[1]{\lVert#1\rVert}

\newcommand{\orb}{\operatorname{Orb}} 
\newcommand{\card}{\operatorname{card}}
\newcommand{\ldens}{\operatorname{\underline{dens}}}
\newcommand{\udens}{\operatorname{\overline{dens}}}
\newcommand{\rank}{\operatorname{rank}}
\newcommand{\spa}{\operatorname{span}}      

\begin{document}

\begin{frontmatter}


\title{Li-Yorke and distributionally chaotic operators\footnote{Dedicated to Professor Antonio Martin\'on  on his
 60th birthday}}

\author[LLG]{T. Berm\'{u}dez}
\ead{tbermude@ull.es}

\author[LLG]{A. Bonilla}
\ead{abonilla@ull.es}

\author[VLC]{F. Mart\'{\i}nez-Gim\'{e}nez}
\ead{fmartinez@mat.upv.es}

\author[VLC]{A. Peris}
\ead{aperis@mat.upv.es}

\address[LLG]{Departamento de An\'{a}lisis Matem\'{a}tico, Universidad de La Laguna, 38271, La Laguna (Tenerife), Spain}

\address[VLC]{IUMPA, Universitat Polit\`{e}cnica de Val\`{e}ncia, Departament de Matem\`{a}tica Aplicada, Edifici 7A,  46022 Val\`{e}ncia, Spain.}

\begin{abstract}
We study Li-Yorke chaos and distributional chaos for operators on Banach spaces. More precisely, we characterize Li-Yorke
chaos in terms of the existence of irregular vectors. Sufficient ``computable'' criteria for distributional and Li-Yorke chaos
are given, together with the existence of dense scrambled sets under some additional conditions. We also obtain certain
spectral properties. Finally, we show that every infinite dimensional separable Banach space admits a distributionally
chaotic operator which is also hypercyclic.
\end{abstract}


\end{frontmatter}

\section{Introduction}

During the last years many  researchers paid attention to the ``wild behaviour'' of orbits governed
by linear operators on infinite dimensional spaces (more especially, on Banach or Fr\'{e}chet spaces). One of the most
significant cases being the \emph{hypercyclicity}, that is, the existence of vectors $x\in X$ such that the
orbit $\orb (T,x):=\{ x,Tx,T^2x,\dots \}$ under a (continuous and linear) operator
$T:X\to X$ on a topological vector space (usually,
Banach or Fr\'{e}chet space) $X$, is dense in $X$. We refer the reader
 to the recent books \cite{BM} and \cite{Grosse-ErdmannPeris10book} for a
 thorough account of the subject. This notion from operator theory joined chaos
 after the definition of Devaney \cite{devaney1989an}, which
 (in our context) requires hypercyclicity and density of the set of periodic points of $T$ in $X$.

The concept of ``chaos'' appeared for the first time in the Mathematical literature in the
paper of Li and Yorke \cite{LY} of mid 70's.

\begin{definition}\label{def_ly}
Let  $(X,d)$ be a metric space. A continuous map $f: X\rightarrow X$
is called {\em Li-Yorke chaotic} if there exists an uncountable
subset $\Gamma\subset X$ such that for every  pair $x,y\in\Gamma$ of
distinct points we have
\[
 \liminf_n d(f^nx,f^ny)=0
 \mbox{ and }
 \limsup_n d(f^nx,f^ny)>0.
\]
In this case, $\Gamma$ is a {\em scrambled} set and $\{x,y\}\subset\Gamma$ a {\em Li-Yorke} pair.
\end{definition}

Li-Yorke chaos was studied for linear operators, for instance, in \cite{DFLM99,FD99}. Incidentally,
every hypercyclic operator $T:X\to X$ on a Fr\'{e}chet space $X$ is Li-Yorke chaotic with respect to
any (continuous) translation invariant metric $d$ on $X$: E.g., fix a hypercyclic vector $x\in X$ and
consider the segment $\Gamma:= \{ \lambda x \ ; \ \abs{\lambda}\leq 1\}$, which is a scrambled set
for $T$.

The notion of distributional chaos was introduced by Schweizer and Smital in \cite{SS04}.
The definition was stated for interval maps with the intention to unify different notions of
chaos for continuous maps on intervals. This concept was widely studied by several authors and
can be formulated in any metric space.  In particular, it is considered for linear operators
defined on Banach or Fr\'{e}chet spaces in \cite{O06,MOP09,HCC09,HTS09,TSZH09}.

For any pair $\{x,y\} \subset X$ and for each $n\in \N$,  the  {\em distributional function} %
\(
 F_{xy}^{n}:\R^+ \rightarrow [0,1]
\) %
is defined by
\[
 F_{xy}^{n}(\tau) = \frac{1}{n}\card  \{ 0\le i\le n-1: d(f^{i}x,f^{i}y)<\tau \}
\]
where $\card\{A\}$ denotes  the cardinality of the set $A$. Define
\begin{align*}
 F_{xy}(\tau)
 &= \liminf_{n\rightarrow \infty} F_{xy}^{n}(\tau)
 \\
 F_{xy}^{*}(\tau)
 &=
 \limsup_{n\rightarrow \infty} F_{xy}^{n}(\tau).
\end{align*}

\begin{definition}[\cite{SS04,O09}] \label{def_dc}
A continuous map $f: X\rightarrow X$ on a metric space $X$ is {\em distributionally chaotic}
if there exist an uncountable subset $\Gamma\subset X$ and $\varepsilon >0$ such that
for every $\tau>0$ and each pair of distinct points
$x,y\in\Gamma$, we have that  $F_{xy}^{*}(\tau) =1$ and $ F_{xy}(\varepsilon)=0$.
The set $\Gamma$ is a {\em distributionally $\varepsilon$-scrambled} set and the pair $\{x,y \}$ a
{\em distributionally chaotic pair}.
Moreover, $f$ exhibits {\em dense distributional chaos} if the
set $\Gamma $ may be chosen to be dense.
\end{definition}

Generally speaking, for any distinct  $x,y\in \Gamma$ the iterations of these points
are arbitrary close and $\varepsilon $ separated alternatively, but additionally there
are time intervals where any of these excluding possibilities is much more frequent than the other.

Given $A\subset\N$, its \emph{upper and lower density} are
defined by
\[
 \udens (A)
 =
 \limsup_n \frac{\card  \{ A\cap [1,n]\}}{n},
 \mbox{ and }
 \ldens (A)
 =
 \liminf_n \frac{\card  \{ A\cap [1,n]\}}{n},
\]
respectively. With these concepts in mind, one can equivalently say
that $f$ is distributionally chaotic on $\Gamma$ if there exists
$\eps>0$ such that for any $x,y\in \Gamma$, $x\neq y$, we have
\[
 \ldens \{ n\in \N \ ; \ d(f^nx,f^ny)<\eps \}=0,
 \mbox{ and }
 \udens \{ n\in \N \ ; \ d(f^nx,f^ny)<\tau \}=1,
\]
for every $\tau >0$.

From now on $X$ will be a Banach space and $T:X\rightarrow X$ a
bounded operator. In this case the associated distance is
$d(x,y)=\norm{x-y}$, with  $x,y\in X$, where $\norm{{\cdot}}$ is the
norm of $X$. We recall the following concept from operator theory.

\begin{definition}[Beauzamy \cite{beau}] \label{def_iv}
A vector $x\in X$ is said to be \emph{irregular} for $T$ if
$\lim \inf_n\norm{T^nx} =0$ and $\lim \sup_n\norm{T^nx}=\infty$.
\end{definition}

Inspired by this definition, and by the notion of distributional chaos,
we consider the following stronger property.

\begin{definition}\label{def_div}
A vector $x\in X$ is said to be \emph{distributionally
irregular} for $T$ if there are increasing sequences of
integers $A=(n_k)_k$ and $B=(m_k)_k$ such that
$\udens(A)=\udens(B)=1$, $\lim_k\norm{T^{n_k}x} =0$ and $\lim_k
\norm{T^{m_k}x}=\infty$.
\end{definition}

\section{Li-Yorke chaotic operators}

We first discuss Li-Yorke chaos for operators with the following
result that establishes the equivalence between Li-Yorke chaos and
the existence of an irregular vector.

\begin{theorem}\label{t_ly_irr}
Let $T:X \to X$ be an operator. The following assertions are
equivalent:
\begin{enumerate}
\item[(i)] $T$ is Li-Yorke chaotic.
\item[(ii)] $T$ admits a Li-Yorke pair.
\item[(iii)] $T$ admits an irregular vector.
\end{enumerate}
\end{theorem}

\begin{proof}
(i) implies (ii) is clear; in order to prove (ii) implies (iii)
suppose that $T$ admits  $y,z\in X$ with
\[
 \liminf_n d(T^ny,T^nz)=0 \mbox{ and } \limsup_n d(T^ny,T^nz)>0.
\]
If we set $x=y-z$ then $\liminf_n\norm{T^nx}=0$ and there is
$\delta>0$ such that $\limsup_n\norm{T^nx}>\delta$. If
$\limsup_n\norm{T^nx}=\infty$ then $x$ is an irregular vector.
Otherwise, $M:=\limsup_n\norm{T^nx}<\infty$.

We observe that $\norm{T}>1$ since, given $n,m\in \N$, $n<m$, with
$\norm{T^nx}<\delta /2$ and $\norm{T^mx}>\delta$, we have that
$\norm{T}^{m-n}>2$. We select a strictly increasing sequence of
integers $(n_k)_k$ such that the sequence of vectors $(T^{n_k}x)_k$
tends to $0$ fast enough so that
\begin{align}
 &\norm{T^{n_{2k}}x}<4^{-k},\notag\\
 & \sum_{j=0}^k \frac{\norm{T^{n_{2j}+n_{2k+1}}x}}{4^j\norm{T}^{n_{2j-1}}\norm{T^{n_{2j}}x}}
 <\norm{T^{n_{2k}}x}, \mbox{ and }\label{eq_ly1}\\
 &\frac{\delta}{4^k\norm{T}^{n_{2k-1}}\norm{T^{n_{2k}}x}} -
 \sum_{j=0}^{k-1} \frac{M}{4^j\norm{T}^{n_{2j-1}}\norm{T^{n_{2j}}x}}
 >k, \label{eq_ly2}
\end{align}
for all $k\in \N$, where $n_{-1}=n_0:=0$. We now set
\[
 u=\sum_{j=0}^\infty \frac{1}{4^j\norm{T}^{n_{2j-1}}\norm{T^{n_{2j}}x}} T^{n_{2j}}x.
\]
We will see that $u$ is an irregular vector for $T$. Indeed,
\begin{align*}
 \norm{T^{n_{2k+1}}u} &\le
 \sum_{j=0}^\infty \frac{\norm{T^{n_{2j}+n_{2k+1}}x}}{4^j\norm{T}^{n_{2j-1}}\norm{T^{n_{2j}}x}}
 \\
 &<\norm{T^{n_{2k}}x}+
 \sum_{j=k+1}^\infty
 \frac{\norm{T^{n_{2k+1}}}\norm{T^{n_{2j}}x}}{4^j\norm{T}^{n_{2j-1}}\norm{T^{n_{2j}}x}}
 \\
 &<4^{-k}+
 \sum_{j=k+1}^\infty \frac{\norm{T}^{n_{2k+1}}}{4^j\norm{T}^{n_{2j-1}}}<2^{-k},
\end{align*}
by \eqref{eq_ly1} and the selection of $T^{n_{2k}}x$, $k \in \N$. On
the other hand, let $(m_k)_k$ be a strictly increasing sequence of
integers such that $\norm{T^{m_k}x}>\delta$, $k\in \N$. Without loss
of generality we suppose that $n_1<m_1<n_2<m_2<\dots$ with
$m_{2k}-n_{2k} $ tending to infinity. By \eqref{eq_ly2} we obtain
that
\begin{align*}
 \norm{T^{m_{2k}-n_{2k}}u}
 &\ge
 \frac{\norm{T^{m_{2k}}x}}{4^k\norm{T}^{n_{2k-1}}\norm{T^{n_{2k}}x}} -
 \sum_{j\neq k} \frac{\norm{T^{m_{2k}-n_{2k}+n_{2j}}x}}{4^j\norm{T}^{n_{2j-1}}\norm{T^{n_{2j}}x}}
 \\
 &>
 \frac{\delta}{4^k\norm{T}^{n_{2k-1}}\norm{T^{n_{2k}}x}}-
 \sum_{j=0}^{k-1} \frac{M}{4^j\norm{T}^{n_{2j-1}}\norm{T^{n_{2j}}x}} -
 \sum_{j=k+1}^\infty \frac{\norm{T}^{m_{2k}-n_{2k}}}{4^j\norm{T}^{n_{2j-1}}}
 \\
 &>
 k-\sum_{j=k+1}^\infty \frac{1}{4^j} > k-4^{-k},
\end{align*}
for all $k\in \N$. Thus $\limsup_n\norm{T^nu}=\infty$ and $u$ is an
irregular vector.

Finally, to see (iii) implies (i), if $u\in X$ is an irregular vector for $T$, then
$S:=\spa \{ u\}$ is a scrambled set for $T$ by a direct application
of the definitions.
\end{proof}

Pr\v{a}jitur\v{a} studied in \cite{praj} some  properties of
operators having irregular vectors.  By \cite{praj} and
Theorem~\ref{t_ly_irr},  we obtain certain properties of Li-Yorke
chaotic operators.

\begin{corollary}\label{ly_spect}
Let $T:X\to X$ be a Li-Yorke chaotic  operator. The following
assertions hold:
\begin{enumerate}
\item
$\sigma(T)\cap\partial \D \ne \emptyset $.
\item
$T^n$ is Li-Yorke chaotic for every $n\in \N$.
\item
$T$ is not compact.
\item $T$ is not normal.
\end{enumerate}
\end{corollary}

\begin{definition}\label{LYC}
An operator $T:X\to X$ satisfies the  \emph{ Li-Yorke Chaos
Criterion} (LYCC) if there exist an increasing sequence of integers
$(n_k)_k$ and a subset $X_0\subset X$ such that
\begin{enumerate}
\item [(a)]
$\displaystyle\lim_{k\rightarrow \infty} T^{n_k}x=0$, $x\in X_0$,

\item [(b)]
$\displaystyle\sup_n \norm{T^n|_Y}=\infty$, where
$Y:=\overline{\spa (X_0)}$ and $T^n|_Y$ denotes the
restriction operator of $T^n$ to $Y$.
\end{enumerate}
\end{definition}

One might expect that the LYCC is a sufficient condition for
Li-Yorke chaos. It turns out that it is in fact a characterization
of this phenomenon.

\begin{theorem}\label{ly_lyc}
Let $T:X\to X$ be an operator. The following are equivalent:
\begin{enumerate}
\item[(i)] $T$ is Li-Yorke chaotic.
\item[(ii)] $T$ satisfies the  Li-Yorke Chaos Criterion.
\end{enumerate}
\end{theorem}

\begin{proof}
Suppose $T$ is Li-Yorke chaotic, by Theorem~\ref{t_ly_irr} we find
an irregular vector $x\in X$. By setting $X_0=\{x\}$,  we verify the
LYCC easily.

Conversely, suppose $T$ satisfies the LYCC, let $(n_k)_k$,
 and $X_0\subset X$ be the respective sequence of
integers and subset of vectors satisfying conditions {(a)} and {(b)}
of Definition~\ref{LYC}. If there is $x\in X_0$ such that $x$ is an
irregular vector, then we are done. Otherwise we observe that every
$u\in \spa (X_0)$ satisfies $\lim_kT^{n_k}u=0$ and
$\sup_n\norm{T^nu}<\infty$. Passing to subsequences of $(n_k)_k$ and
$(m_k)_k$, if necessary, by property {(b)} and since
$\sup_n \norm{T^n|_Y}=\infty$,
  we can obtain a sequence $(u_j)_j$
of normalized vectors in $\spa (X_0)$   such that

 \begin{enumerate}
 \item[(a)'] $\norm{T^{n_k} u_j}<\frac{1}{j}$, $j=1,\dots,k$, $k\in \N$, and
 \item[(b)'] $\norm{T^{m_j} u_j} > 3^j M_{j-1}$, $j>1$
 \end{enumerate}
 where $M_k:=\sup \{\norm{T^nu_i} \ ; \ i=1,\dots ,k, \ n\geq 0 \}<\infty$, $k\in \N$.
 Without loss of generality, we may suppose that $m_1<n_1<m_2<n_2<\dots $. We select any infinite
 subset $I\subset \N$ such that, for each $j\in \N$, if $i\in I$ with $i>j$, then $2^i>2^j\norm{T}^{n_j}$.
The vector
\[
 u:= \sum_{i\in I} \frac{1}{2^i} u_i
\]
is well-defined since the series is convergent. We will see that
$u$ is an irregular vector for $T$.
On one hand,
\begin{align*}
 \norm{T^{m_j}u}
 &\ge
 \frac{1}{2^j}\norm{T^{m_j}u_j} -
 \sum_{i\in I, \ i\neq j} \frac{1}{2^i}\norm{T^{m_j}u_i}
 \\
 &>
 \frac{3^jM_{j-1}}{2^j} - \sum_{i\in I, \ i< j} \frac{M_{j-1}}{2^i} -
 \sum_{i\in I, \ i> j} \frac{1}{2^i}\norm{T^{m_j}u_i}
 \\
 &>
 \left(\frac{3^j}{2^j} -1\right) M_{j-1}- \sum_{ i\geq  j} \frac{1}{2^i} \ \
 \underset{j\to \infty , \ j\in I}{\longrightarrow} \ +\infty .
\end{align*}

On the other hand,
\[
 \norm{T^{n_j}u} \leq \sum_{i\in I, \ i\leq j} \frac{j^{-1}}{2^i} +
 \sum_{i\in I, \ i> j} \frac{1}{2^i}\norm{T^{n_j}u_i} <
 \frac{1}{j}+\frac{1}{2^{j-1}}, \ j\in I,
\]
which shows that $u$ is an   irregular vector.
\end{proof}

\section{The strong criterion for distributional chaos and spectral properties}
\label{sec:SCDC-SP}

The following  criterion for distributional chaos was introduced in
\cite{HCC09}. Since, for several reasons that will be clarified
soon, this criterion is somehow very restrictive, we will call it
the ``strong'' criterion. In this section we will study the spectral
properties of operators that satisfy this criterion.

\begin{definition}\label{SCDC}
An operator $T:X\rightarrow X$ satisfies the \emph{Strong
 Distributional Chaos Criterion} (SDCC), if there is a constant
$r >1$ such that for any  $m\in \N$, there exists $x_{m}\in
X\setminus\{0\}$ satisfying
\begin{enumerate}\setlength{\itemsep}{0pt}
\item[(i)]
$\displaystyle \lim_{k\rightarrow \infty} \|T^{k}x_{m}\|=0$,
\item[(ii)]
$\|T^{i}x_{m}\|\ge r ^{i} \|x_{m}\|$ for  $i=1,2, \ldots, m$.
\end{enumerate}
Such $r$ is said to be a {\em SDCC-constant} for the operator $T$.
\end{definition}

\begin{theorem}[{\cite[Theorem 3.3]{HCC09}}]\label{SCDC_dc}
Let $T:X\rightarrow X$ be an operator. If $T$ satisfies the Strong
Distributional Chaos Criterion, then $T$ is distributionally
chaotic.
\end{theorem}

By $\D$ we mean the open unit disc, its boundary is $\partial{\D}$ and the complement of
the closed unit disc will be written as $\overline{\D}^c$. We denote by $r(T)$ the spectral radius
of $T$.

\begin{proposition}\label{SCDC_spect}
Let $T:X\rightarrow X$ be an operator. The following properties
hold:
\begin{enumerate}\setlength{\itemsep}{0pt}
\item[(a)]
If there  exists $r>1$ such that for all $m\in\N$ there exists
$x_m\in X\setminus\{0\}$ with
\begin{equation}\label{r}
 \|T^ix_m\|\geq r^i\| x_m\|, \text{ for all } i=1, \ldots , m,
\end{equation}
then $r(T)\geq r$.

\item[(b)]
\cite[Lemma 6.4 (b)]{FMM02} If $\sigma(T) \subset \overline{\D}^c$ and
$\lim_{k\to\infty}\|T^kx\|=0$, then $x=0$.
\end{enumerate}
\end{proposition}

\begin{proof}
(a) Assume that $r(T)<r$. Let $\varepsilon >0$ be such that
$r(T)+\varepsilon <r$. Then there exists $m\in\N$ such that
for all $n\geq m$ we have that $\|T^n\|^{1/n} \leq r(T)+\varepsilon
<r$. Moreover, by (\ref{r}) for $m+1$ there exists $x\in
X\setminus\{0\} $ such that $\| T^ix\|\geq r^i\|x\|$ for $i=1,
\ldots, m+1$. Then $\|T^{m+1}\frac{x}{\|x\|}\|\geq r^{m+1}$, so
$\|T^{m+1}\|^{1/{m+1}}\geq r$, which is a contradiction.
\end{proof}

\begin{corollary}\label{spect_rad}
If $T:X\rightarrow X$ satisfies the Strong Distributional Chaos
Criterion with SDCC-constant $r$, then $r(T)\ge r$.
\end{corollary}

\begin{theorem}\label{SCDC_spect2}
 If  $T$ satisfies the Strong Distributional Chaos Criterion with
SDCC-constant $r$, then the following properties hold:
\begin{enumerate}
\item[(a)]
For any $r_0$ with $1\le r_0 \le r$ we have $\sigma(T)\cap
\partial D(0,r_0) \ne\emptyset$.
\item[(b)]
There are not disjoint  closed subsets $F_1$ and $F_2$ of $\C$
such that  $F_1\subset D(0,r)$ and $F_2\subset \overline{\D}^c$
with $F_1\cup F_2= \sigma (T)$.
\end{enumerate}
\end{theorem}

\begin{proof}
(a) Assume that there exists $r_0\le r$ such that $\sigma(T)\cap
\partial D(0,r_0) = \emptyset$. Take the spectral decomposition
$\sigma_1= D(0,r_0) \cap \sigma(T) $ and
$\sigma_2= \overline{D(0,r_0)}^c \cap \sigma(T)$. Let $X_i$ and
$T_i$ with $i=1,2$ be given by the  above spectral  decomposition such
that $X=X_1 \oplus X_2$, $T=T_1 \oplus T_2$ with
$\sigma_i=\sigma(T_i)$, $i=1,2$. Let $m\in\N$. Then there
exists $x_m=x_m^1\oplus x_m^2\in X\setminus \{0\}$ such that
$\|T^kx_m\|=\|T_1^kx_m^1\|+\|T^k_2x_m^2\| \to 0$ as $k$ tends to
infinity and  $\|T^ix_m\|\geq r^i\| x_m\| $ for $i=1, \ldots, m$.
Using the same argument as in \cite[Lemma 6.4 (b)]{FMM02}  we have
that $\|x_m^2\| \leq \|T^{k}_2x_m^2\|$, which tends to $0$ as $k\to \infty$,
that is, $x_m^2=0$. Hence $x_m=x_m^1$ and
$\|T^kx_m\|=\|T_1^kx_m^1\|\geq r^i\| x_m\| $ for $i=1, \ldots, m$.
Henceforth, $r(T_1)\geq r$ but this is a contradiction since $\sigma(T_1)
\subset D(0,r_0)$.

(b) Let $F_1\subset D(0,r)$ and $F_2\subset \overline{\D}^c$ be
closed subsets of $\sigma(T)$  such that $F_1\cap F_2=\emptyset$ and
$F_1\cup F_2 = \sigma(T)$. Let $X_i$ and $T_i$ with $i=1,2$ be given
by this spectral decomposition, i.e., $X=X_1 \oplus X_2$, $T=T_1
\oplus T_2$ with $F_i=\sigma(T_i)$, $i=1,2$. Let $m\in \N$. Then
there exists $x_m=x_m^1\oplus x_m^2\in X\setminus \{0\}$ such that
$\|T^kx_m\|=\|T_1^kx_m^1\|+\|T^k_2x_m^2\| \to 0$ as $k$ tends to
infinity and  $\|T^ix_m\|\geq r^i\| x_m\| $ for $i=1, \ldots, m$. By
part (b) of Proposition \ref{SCDC_spect} we have that $x_m^2=0$.
Then by part (a) of Proposition \ref{SCDC_spect} we obtain that
$r(T_1)\geq r$ which is a contradiction since $\sigma(T_1)=\sigma_1\subset
D(0,r)$.
\end{proof}

An operator $S:X\to X$ is called {\em strictly singular} if for
every infinite dimensional subspace $M$ of $X$, the restriction of
$S$ to $M$ is not a homeomorphism. An example of strictly singular
operator is the compact operators.

Notice that a small compact perturbation of the unit operator could
be distributionally chaotic \cite[Proposition 3.5]{HTS09}. Also, in
\cite[Proposition 3.3]{HTS09} is proven that any compact
perturbation of scalar operator can not satisfy the SDCC. In the
next result we improve that property, that is, if $S$ is a strictly
singular operator, then $S+\lambda I$ does   not satisfy the SDCC.

\begin{corollary}\label{SS_SCDC}
If $S:X\to X$ is a  strictly singular  operator  and $\lambda \in
\C$, then $\lambda I+S$ does not satisfy  the SDCC.
\end{corollary}

\begin{proof}
It is well know that the spectrum of  strictly singular operator is
at most a countable  sequence of eigenvalues that converges to zero
as its only limit point. Hence the result is a consequence of part
(b) of Theorem \ref{SCDC_spect2}.
\end{proof}

A  \emph{hereditarily indecomposable}  Banach space  is an infinite
dimensional space such that no subspace can be written as a
topological sum of two infinite dimensional subspaces. W.T. Gowers
and B. Maurey constructed the first example of a hereditarily
indecomposable space \cite{GM93}. In particular, they proved that
every  operator $T$ defined on a hereditarily indecomposable space
$X$ can be written as $T=\lambda I+S$, where $\lambda \in \C$ and
$S$ is a strictly singular operator.

The class of hereditarily indecomposable spaces appeared for the
first time in relation with hypercyclicity and chaos in
\cite{BMP01}, where  it was proved that some complex separable
Banach spaces admit no Devaney chaotic operator.

\begin{corollary}\label{HI_SCDC}
There are no  operators satisfying the SDCC  on any hereditarily
indecomposable Banach space.
\end{corollary}

\section{Distributionally chaotic operators}

Our purpose in this section is to study distributional chaos in connection with
the existence of distributionally irregular vectors, and to give useful criteria
for distributional chaos. First of all, we obtain the analogous
of one of the implications in Theorem~\ref{t_ly_irr}.

\begin{proposition}\label{div_dc} If  $T:X \to X$
admits a distributionally irregular vector $x\in X$, then  $T$ is
distributionally chaotic.
\end{proposition}

\begin{proof}
If $x\in X$ is a distributionally irregular vector for $T$, then
there exist increasing sequences of integers $A=(n_k)_k$ and
$B=(m_k)_k$ such that $\udens (A)=\udens (B)=1$, $\lim_k
\norm{T^{n_k}x}=0$ and $\lim_k\norm{T^{m_k}x}=\infty$. Let $S:=\spa
\{ x\}$. If $y,z\in S$, $y\neq z$, then $y-z=\alpha x$ with $\alpha
\neq 0$. Therefore,
\begin{align*}
 \norm{T^{n_k}y-T^{n_k}z}
 &=\abs{\alpha} \norm{T^{n_k}x}\stackrel{k\to\infty}{\longrightarrow} 0, \mbox{ and }
 \\
 \norm{T^{m_k}y-T^{m_k}z}
 &=\abs{\alpha} \norm{T^{m_k}x}\stackrel{k\to\infty}{\longrightarrow} \infty ,
\end{align*}
which yields that $S$ is a distributionally $\varepsilon$-scrambled
set for $T$, for every $\varepsilon >0$.
\end{proof}

\begin{problem}
Does every distributionally chaotic  operator $T:X\to X$ admit a
distributionally irregular vector?
\end{problem}

The following result from \cite{HCC09} yields a sufficient condition for distributional chaos,
less restrictive than the SDCC.

\begin{theorem}[\cite{HCC09}]\label{wdcc}
Let $T:X\rightarrow X$ be an operator. If for any sequence of
positive numbers $(C_{m})_m$ increasing to $\infty$, there exists
$(x_{m})_{m}$ in $X \setminus \{ 0 \} $ satisfying
\begin{enumerate}
\item [(a)]
$\displaystyle \lim_{n\rightarrow \infty} T^{n}x_{m}=0$, and

\item [(b)]
there is a sequence of positive integers $(N_{m})_m$ increasing to
$\infty $, such that
\[
 \lim _{m\rightarrow\infty}
 \frac{1}{N_{m}}\card  \{ 0\le i < N_{m} \ ; \ \|T^{i}x_{m}\|
 \ge C_{m}\|x_{m}\|\}=1,
\]
\end{enumerate}
then $T$ is distributionally chaotic.
\end{theorem}

We introduce a variation of the above criterion by allowing to have
different sequences of vectors  in  part (a) and (b) of the above
theorem.

\begin{definition}\label{CDC}
An operator $T:X\to X$ satisfies the  \emph{Distributional Chaos
Criterion } (DCC) if there exist sequences $(x_{m})_{m}$ and
$(y_{m})_{m}$ in $X\setminus \{ 0 \}$ with $y_{m}\in \overline{\spa
\{x_{k} \ ; \ k \in \N\}  }$ satisfying
\begin{enumerate}
\item [(a)]
$\displaystyle\lim_{n\rightarrow \infty} T^{n}x=0$, $x\in X_0$, and

\item [(b)] there is a sequence of positive integers $(N_{m})_m$ increasing
to $\infty $, such that
\[
 \lim _{m\rightarrow \infty}\frac{1}{N_{m}}\card  \{ 0\le i < N_{m} \ ; \ \|T^{i}y_{m}\| \geq m\|y_{m}\|\}=1\;.
\]

\end{enumerate}
\end{definition}

Observe that our criterion has, a priori, weaker requirements than
the criterion of Cao, Cui and Hou in Theorem~\ref{wdcc}. We will
show that they are actually equivalent.

\begin{theorem}\label{wdcc_dcc}
Let $T:X\rightarrow X$ be an operator. The following properties are
equivalent:
\begin{enumerate}
\item[(i)]
$T$ satisfies the hypothesis of Theorem~\ref{wdcc}.

\item[(ii)]
$T$ satisfies the Distributional Chaos Criterion.
\end{enumerate}
\end{theorem}

\begin{proof}
We only need to prove that (ii) implies (i). By (ii) there exists
$({x}_{m})_{m}\subset \spa (X_0) \setminus \{ 0 \} $ that satisfies
the condition (a) of Theorem~\ref{wdcc} by linearity, and condition
(b) by density.
\end{proof}

Under the assumptions of the DCC we can ensure the existence of
distributionally irregular vectors.

\begin{proposition}\label{dcc_div}
If $T:X\to X$ satisfies the DCC, then $T$ admits a distributionally
irregular vector.
\end{proposition}

\begin{proof}
By Theorem \ref{wdcc_dcc}, passing to subsequences if necessary, we
find a sequence $(x_m)_m$ of normalized vectors in $X$ and an
increasing sequence $(N_m)_m$ of integers with $N_m-N_{m-1}$ tending
to infinity such that
\begin{align}
 &
 \frac{1}{N_{m}}\card  \{ 0\le i < N_{m}: \|T^{i}x_{m}\|
 > m\norm{T}^{N_{m-1}}\}> 1-\frac{1}{m}, \label{eq1}
 \\
 &
 \frac{1}{N_{m}}\card  \{ 0\le i < N_{m}: \|T^{i}x_{k}\|
 <\frac{1}{m} \} > 1-\frac{1}{m^2}, \
 k=1,\dots, m-1, \label{eq2}
\end{align}
where $N_0:=1$.

Since our hypothesis obviously imply that $\norm{T}>1$, the series
\[
 x:=\sum_k \frac{1}{\norm{T}^{N_{2k-1}}} x_{2k}
\]
is convergent, thus $x\in X$. We will show that $x$ is a
distributionally  irregular vector for $T$. Indeed,
\begin{align*}
 \norm{T^ix}
 &\ge
 \frac{1}{\norm{T}^{N_{2m-1}}} \norm{T^ix_{2m}}
 -
 \sum_{k\neq m} \frac{1}{\norm{T}^{N_{2k-1}}} \norm{T^ix_{2k}}
 \\
 &>
 2m-\frac{1}{2m} \sum_{k<m} \frac{1}{\norm{T}^{N_{2k-1}}}- \sum_{k>m}
 \frac{1}{\norm{T}^{N_{2k-1}-N_{2m}}}
 \underset{m\to\infty}{\longrightarrow} \infty
\end{align*}
if $i< N_{2m}$, $\|T^{i}x_{2m}\| > 2m\norm{T}^{N_{2m-1}}$, and
$\|T^{i}x_{2k}\| <\frac{1}{2m}$, $k<m$. Conditions (\ref{eq1}) and (\ref{eq2})
above imply
\[
 \frac{\card \{ 0\leq i<N_{2m} \ ; \ \|T^{i}x_{2m}\| >
 2m\norm{T}^{N_{2m-1}}, \ \  \|T^{i}x_{2k}\| <\frac{1}{2m}, \
 k<m\}}{N_{2m}}  >\frac{m-1}{m},
\]
and we obtain an increasing sequence of integers $B=(m_k)_k$ such
that $\udens (B)=1$ and $\lim_{k\to \infty} \norm{T^{m_k}x}=\infty$.
On the other hand,
\[
 \norm{T^ix} \leq \sum_k \frac{1}{\norm{T}^{N_{2k-1}}}
 \norm{T^ix_{2k}}<\frac{1}{2m+1} \sum_{k\leq m}
 \frac{1}{\norm{T}^{N_{2k-1}}}+ \sum_{k>m}
 \frac{1}{\norm{T}^{N_{2k-1}-N_{2m}}}
 \underset{m\to\infty}{\longrightarrow} 0,
\]
if $i< N_{2m+1}$,  and $\|T^{i}x_{2k}\| <\frac{1}{2m+1}$, $k\leq m$.
Condition (\ref{eq2}) above implies that
\[
 \frac{1}{N_{2m+1}} \card \{ 0\leq i<N_{2m+1} \ ; \ \|T^{i}x_{2k}\|
 <\frac{1}{2m+1}, \ k=1, \dots , m\} >1-\frac{1}{m},
\]
which gives an increasing sequence $A=(n_k)_k$ of integers such that
$\udens (A)=1$ and $\lim_k \norm{T^{n_k}x}=0$, concluding the
result.
\end{proof}

Observe that Propositions \ref{div_dc} and \ref{dcc_div} provide an
alternative proof of Theorem~\ref{wdcc}.

\begin{remark}
Not every distributionally chaotic operator satisfies the
Distributional Chaos Criterion. Indeed, let us consider a weighted
forward shift $F_w:\ell^2\to\ell^2$, defined as $(x_1,x_2,\dots
)\mapsto (0,w_1x_1,w_2x_2,\dots )$, where the sequence of weights
$w=(w_k)_k$ consists, alternatively, of sufficiently large blocks of
$2$'s and blocks of $1/2$'s such that the vector $e_1=(1,0,\dots )$
is a distributionally irregular vector. On the other hand, if
$\lim_kF^k_wx=0$ then $x=0$.
\end{remark}

If we impose that the orbits converge to $0$ on a dense subset, then
the DCC can be characterized in terms of the existence of
distributionally irregular vectors.

\begin{corollary}\label{Cor2}
Let $T:X\rightarrow X$ be an operator such that there exists a dense
set $D$ with $\displaystyle\lim_{n\rightarrow \infty} T^{n}x=0$, for
all $ x\in D$. The following properties are equivalent:
\begin{enumerate}
\item[(i)]
$T$ satisfies the  Distributional Chaos Criterion.

\item[(ii)] $T$ admits a distributionally irregular vector.
\item[(iii)]
There exist a sequence  $(y_{m})_{m}$ in $X\setminus\{0\}$  and a
sequence of positive integers $(N_{m})_m$ increasing to $\infty $,
such that
\[
 \lim _{m\rightarrow \infty}\frac{1}{N_{m}}\card  \{ 0\le i < N_{m} \ ; \
 \|T^{i}y_{m}\| \geq m\|y_{m}\|\}=1.
\]
\end{enumerate}
\end{corollary}

We can get stronger results under the assumptions above, which
depend on the existence of large scrambled sets consisting of
(distributionally) irregular vectors.

\begin{definition}\label{irreg_manif}
A linear manifold $Y\subset X$ is a \emph{(distributionally)
irregular manifold} for $T:X\to X$ if every non-zero vector $y\in
Y\setminus \{ 0\}$ is a (distributionally) irregular vector for $T$.
\end{definition}

Certainly, a (distributionally) irregular manifold is a
(distributionally) scrambled set.

\begin{theorem}\label{DDC}
Let $T:X\rightarrow X$ be an operator. If there exists a dense
subset $X_{0}\subset X$ such that $\lim_{n\rightarrow\infty}
T^{n}x=0$, for each $ x\in X_{0}$ and  $T$ admits a distributionally
irregular vector, then $T$ is admits a dense distributionally
irregular manifold.
\end{theorem}

\begin{proof}
We consider a dense sequence $(y_n)_n$ in $X_0$. By
Proposition~\ref{dcc_div} there is a distributionally irregular
vector $x\in X$ for $T$ which can be written as a series
\[
 x=\sum_{k=1}^\infty x_k
\]
with the properties specified in the proof. We select a countable
collection $(\gamma_m)_m=((\gamma_{m,k})_k)_m$ of sequences of $0$'s
and $1$'s such that each sequence $\gamma_m$ contains an infinite
number of $1$'s and
\[
 \gamma_{m,k}=1 \Longrightarrow \gamma_{n,k}=0, \ \ \forall m\neq n,
 \ \ \forall k\in \N.
\]
We set the sequence of vectors $u_m=\sum_k \gamma_{m,k} x_k$, $m\in
\N$, which, by following the proof of Proposition~\ref{dcc_div}, are
distributionally irregular vectors since each $\gamma_m$ contains an
infinite number of $1$'s. Define now
\[
 z_m=y_m+\frac{1}{m} u_m , \ \ m\in \N.
\]

Since $(y_n)_n$ is dense in $X$ and the $u_m$'s are uniformly
bounded, we get that the sequence $(z_m)_m$ is dense in $X$. We set
$Y=\spa\{z_m \ ; \ m\in \N\}$, which is a dense subspace of $X$.
If $y\in Y\setminus \{ 0\}$, then we can write
\[
 y=y_0+\sum_k \rho_k x_k ,
\]
where $y_0\in X_0$ and the sequence of scalars $\rho=(\rho_k)_k$
only takes a finite number of values, being non-zero an infinite
number of $\rho_k$'s. This means, by following again the proof of
Proposition~\ref{dcc_div}, that
\[
 v:=\sum_k \rho_k x_k
\]
is a distributionally irregular vector for $T$. Since $y=y_0+v$ and
$\lim_kT^k y_0=0$, we get that $y$ is also a distributionally
irregular vector for $T$, and $Y$ is therefore a distributionally
scrambled set for $T$.
\end{proof}

An adaptation of the above argument, taking into account the proof Theorem~\ref{ly_lyc},
yields a useful sufficient condition for dense Li-Yorke chaos.

\begin{theorem}\label{dly}
Let $T:X\rightarrow X$ be an operator. If $\sup_n\norm{T^n}=\infty$
 and there  is a
dense subset $X_{0}\subset X$ such that $\lim_{k\rightarrow\infty}
T^{n}x=0$ for each $ x\in X_{0}$, then $T$ admits a dense irregular
manifold.
\end{theorem}

A class of operators for which Theorems~\ref{DDC} and \ref{dly} are particularly
interesting and easy to apply is the class of operators with dense generalized kernel.

\begin{corollary}\label{dgk}
Let $T:X\to X$ be an operator whose generalized kernel $\bigcup_n
\ker T^n$ is dense in $X$.
\begin{enumerate}
\item[(i)]
If $T$ has a distributionally irregular vector, then $T$ admits a
dense distributionally irregular manifold.

\item[(ii)] If $\sup_n\norm{T^n}=\infty$, then  $T$ has a dense irregular manifold.
\end{enumerate}
\end{corollary}

To apply Corollary~\ref{dgk} to an important class of operators, we
will consider weighted backward shifts $B_w:X\to X$, $(x_0,x_1,\dots )\mapsto(w_1x_1,w_2x_2,\dots)$ on
$ X=\ell^p$, $1\leq p<\infty$,  or $X=c_0$,
where $w=(w_i)_i$ is a bounded sequence of non-zero weights, and the
(unweighted) backward shift $B:X\to X$, $(x_0,x_1,\dots)\mapsto(x_1,x_2,\dots)$
on the weighted spaces
\begin{align*}
 X
 &=\ell^p(v):=
 \left\{x=(x_i)_i \ ; \ \norm{x}:=\left(\sum_i v_i\abs{x_i}^p\right)^{1/p}<\infty\right\},\
 1\leq p<\infty , \mbox{ or}
 \\
 X
 &=
 c_0(v):=\{x=(x_i)_i \ ; \ \lim_i v_ix_i=0, \ \norm{x}:=\sup_iv_i\abs{x_i} \},
\end{align*}
 where $v=(v_i)_i$ is a sequence of strictly positive
weights such that $\sup_i v_i/v_{i+1}<\infty$ in order to have a
well-defined backward shift bounded operator.

\begin{proposition}\label{wbs}
The following assertions hold:
\begin{enumerate}
\item[(i)]
$B$ is Li-Yorke chaotic if and only if $M_v:=\sup\{v_n/v_m;\ n\in
\N, m>n\}=\infty$. In this case, $B$ admits a dense irregular
manifold.

\item[(ii)]
$B_w$ is Li-Yorke chaotic if and only if
$M_w:=\sup\left\{\prod_{k=n}^m \abs{w_k};\ n\in \N,
m>n\right\}=\infty$. In this case, $B_w$ admits a dense irregular
manifold.
\end{enumerate}
\end{proposition}

\begin{proof}
Suppose that $B:\ell^p(v)\to \ell^p(v)$ is Li-Yorke chaotic, and let
$x=(x_i)_i$ be an irregular vector for $B$.
Since
\[
 \norm{B^nx}^p=\sum_{k=0}^\infty v_k\abs{x_{k+n}}^p=\sum_{k=0}^\infty
 \left( \frac{v_k}{v_{k+n}}\right) v_{k+n}\abs{x_{k+n}}^p\leq M_v\norm{x}^p,
\]
then $M_v<\infty$ implies that $\sup_n\norm{B^nx}<\infty$, which is a contradiction.
Conversely, if $M_v=\infty$, we find increasing sequences of integers $(n_k)_k$, $(m_k)_k$,
$n_k<m_k$, $k\in \N$, such that $\lim_k (m_k-n_k)=\infty$ and $ v_{n_k}/v_{m_k}>3^k$, $k\in \N$,
by the definition of $M_v$ and the condition on $v$ for continuity of $B$.
We define the vector $x=(x_i)_i\in \ell^p (v)$ by $x_i=(2^kv_{m_k})^{-1/p}$ if $i=m_k$, and
$x_i=0$ otherwise. We have
\[
\norm{B^{m_k-n_k}x} \geq v_{n_k} \abs{x_{m_k}}^p=\left( \frac{v_{n_k}}{v_{m_k}}\right)
\frac{1}{2^k} > \left( \frac{3}{2} \right) ^k \ \ \underset{k\to \infty}{\longrightarrow} +\infty .
\]
This shows that $B$ is Li-Yorke chaotic, and it admits a
dense linear manifold $Y\subset X$ such that $Y$ is a
scrambled set for $B$ by Corollary~\ref{dgk}. The case $X=c_0$ is analogous.

(ii) is a consequence of (i) if we proceed by conjugation with a suitable
diagonal operator (see, e.g., \cite{MGP02}).
\end{proof}

Another consequence of Theorem~\ref{DDC} is the following easy sufficient condition
for dense distributional chaos.

\begin{corollary}\label{orbit_to_infty}
If an operator $T:X\to X$ satisfies that
\begin{enumerate}
\item there exist an increasing sequence of integers $B=(m_k)_k$ with $\udens (B)=1$,  $y\in X$
satisfying $\lim_{k\to \infty} \norm{T^{m_k}y}=\infty$, and
\item a dense subset $X_{0}\subset X$ such that
$\lim_{n\rightarrow\infty} T^{n}x=0$, for each $ x\in
X_{0}$,
\end{enumerate}
then  $T$ admits a dense distributionally irregular manifold.
\end{corollary}

A this point we want to recall recent result of M\"{u}ller and Vrsovsky.

\begin{theorem}[\cite{MV09}]\label{MV_orbit_to_infty}
Let $T_n:X\to X$, $n\in \N$, be a sequence of operators. Suppose that one of the following conditions
is satisfied:
\begin{enumerate}
\item[(i)] either $\sum_{n=1}^\infty \frac{1}{\norm{T_n}} <\infty ;$

\item[(ii)] or $X$ is a complex Hilbert space and
$\sum_{n=1}^\infty \frac{1}{\norm{T_n}^2} <\infty .$
\end{enumerate}
Then there exists a point $y\in X$ such that $\lim_{n\to \infty} \norm{T_ny}=\infty$.
\end{theorem}

Corollary~\ref{orbit_to_infty} and Theorem~\ref{MV_orbit_to_infty} combine nicely to obtain a powerful
sufficient condition for dense distributional chaos.

\begin{corollary}\label{MV_ddc}
Let $T:X\to X$ be an operator such  that
 there exist a dense subset $X_{0}\subset X$ with
$\lim_{n\rightarrow\infty} T^{n}x=0$, for each $ x\in
X_{0}$, and an increasing sequence of integers $B=(m_k)_k$ with $\udens (B)=1$ satisfying
that
\begin{enumerate}
\item[(i)] either
$\sum_{k=1}^\infty \frac{1}{\norm{T^{m_k}}} <\infty ;$

\item[(ii)] or $X$ is a complex Hilbert space and
$\sum_{k=1}^\infty \frac{1}{\norm{T^{m_k}}^2} <\infty .$
\end{enumerate}
Then  $T$ has  a dense distributionally irregular manifold.
\end{corollary}

As consequence of Corollary~\ref{MV_ddc}, we obtain the following
property.

\begin{corollary}\label{DDCC}
Let $T:X\rightarrow X$ be an operator. If there exist a dense set
$X_{0}$ such that $\lim_{n\rightarrow \infty} T^{n}x=0$, for each $
x\in X_{0}$ and $r(T)>1$, then $T$ admits a dense distributionally
irregular manifold.
\end{corollary}

\begin{definition}[\cite{CD78}]
Let $H$ be a complex separable Hilbert space. For $\Omega$ a
connected open subset of $\C$ and $n$ a positive integer, let
${\textrm{B}}_{n}(\Omega)$ denote the set of  operators $T$ in $H$
which satisfy
\begin{enumerate}
\item[(a)] $\Omega \subset \sigma (T)$,

\item[(b)] $\rank (T-wI) = H$ for $w\in \Omega$,

\item[(c)] $\bigvee \ker _{ w\in \Omega}(T-wI) =H$,

\item[(d)] $\dim \ker (T-wI) =n$ for $w\in \Omega$.
\end{enumerate}
\end{definition}

\begin{corollary}
Let $T\in {\textrm{B}}_{n}(\Omega)$. If $\Omega\cap\T\ne\emptyset$,
then $T$ has a dense distributionally irregular manifold.
\end{corollary}

\begin{proof}
 Let $U$ be a
connected open subset of $\Omega\cap \D$. Then $X_{0}:= \spa
\bigcup_{ w\in U}\ker (T-wI)$ is dense in $H$ and satisfies that
$\lim_{n\rightarrow \infty} \|T^{n}x\|=0$,  for all $x\in X_{0}$ and
$r(T)>1$. Hence the conclusion follows by an application of
Corollary~\ref{DDCC}.
\end{proof}

\section{Existence of distributionally chaotic operators}

The study of general conditions under which a space $X$ admits operators
with certain wild behaviour has attracted the interest of many researchers.
Ansari \cite{Ans97} and Bernal-Gonz\'{a}lez
\cite{Ber99} independently proved that any separable infinite dimensional Banach space $X$
supports a hypercyclic operator.  This result was extended to Fr\'echet spaces
by Bonet and Peris \cite{BonetPeris98}. A generalization for polynomials is given in \cite{MGP09}.
Other existence results for related properties and for $C_0$-semigroups can be found, for instance, in
\cite{BBM03,BGE07,BBCP05,S08,S10}.

In this section we want to establish the existence of distributionally chaotic operators on
arbitrary infinite dimensional and separable Banach spaces. The following Lemma of Shkarin \cite{S08}
and its consequences will be key to obtain our existence result.

\begin{lemma}[\cite{S08}]\label{shk}
Let $X$ be a separable infinite dimensional Banach space, let
$(b_{k})_{k}$ be a sequence of numbers in
$[3,+\infty[$ such that $ b_{k}\rightarrow \infty$ as $k\rightarrow
\infty$ and $(N_{k})_{k}$ be a strictly increasing
sequence of positive integers such that $N_{0} =0$ and $N_{k+1}-
N_{k}\ge 2$ for each $k\in\N$. Then there exists a
biorthogonal sequence $((y_{k}, f_{k}))_{k}$ in $X\times
X^* $ such that

(B1) $\|y_{k}\|=1$ for each $k\in\N$;

(B2) $\spa \{y_{k} :k\in\N\}$ is dense in $X$;

(B3) $\|f_{n_{k}}\| \le b_{k}$ for each $k\in\N$;

(B4) $\|f_{j}\|\le 3 $ if $j\in\N\setminus\{N_{k}: k\in\N\}$;

(B5) for any $k\in\Z$ and any $ c_{j} \in \K $ with  $ N_{k} +1 \le j \le N_{k+1}-1$
\[
 \frac{1}{2} \norm{ \sum _{j = N_{k} +1}^{N_{k+1}-1} c_{j}y_{j}} \le
 \left( \sum _{j = N_{k} +1}^{N_{k+1}-1}
 |c_{j}|^2\right)^{\frac{1}{2}}
 \le
 2\norm{ \sum _{j = N_{k}+1}^{N_{k+1}-1} c_{j}y_{j}}
\]
\end{lemma}

In Section~\ref{sec:SCDC-SP} we showed that there are no operators
satisfying the Strong   Distributional Chaos  Criterion on certain
Banach spaces. In the following result
we give a positive answer for distributionally chaotic operators.

\begin{theorem} In every infinite dimensional separable Banach space there exists a
hypercyclic and distributionally chaotic  operator which admits a
dense distributionally irregular manifold.
\end{theorem}

\begin{proof}
By \cite[Lemma 2.5]{S08}, given a sequence $(w_{n})_{n}$ of
positive weights in $\ell^2$ there exists $T: X\rightarrow X$
satisfying that $Ty_{0} =0$ and $Ty_{n+1}=w_{n}y_{n}$ with
$\|y_{n}\| =1$ where $(y_{n})_n $ is given by Lemma~\ref{shk}.

We need to compute $\|(I+T)^i y_{j}\|$, that is
\begin{align*}
 \|(I+T)^i y_{j}\|&= \norm{\sum_{k=0}^{i} \binom{i}{k} T^k y_{j}} \\
 &=
 \norm{  y_{j}+ \binom{i}{1} T y_{j}+ \cdots +\binom{i}{i} T^i y_{j}} \\
 &=\norm{  y_{j}+ i w_{j-1}y_{j-1}+ \cdots +
 w_{j-1}\cdots w_{j-i} y_{j-i}}
\end{align*}

If $j = N_{k+1}-1$ and $j-i \ge N_{k} +1$, that is, $i\le N_{k+1} - N_{k}-2$,
then
\[
 \|(I+T)^i y_{j}\|\ge \frac{1}{2}i w_{j-1}
\]
by condition (B5) in Lemma~\ref{shk}. Taking
$w_{k}=k^{-\frac{2}{3}}$, $N_{0}:=0$ and $N_{k}=(k+1)!+1$ for $k\ge
1$,  then  for $i= 3k([((k+2)!)^\frac{2}{3}]+1), \cdots,
(k+2)!-(k+1)! -2$
\[
 \|(I+T)^i y_{j}\|\ge \frac{1}{2}i w_{j-1}\ge
 \frac{1}{2}\frac{3k([((k+2)!)^\frac{2}{3}]+1) )}{((k+2)!-1)^{\frac{2}{3}}}>
 k+1,
\]
for $j=(k+2)!=N_{k+1}-1$.
 Thus, for  $u_{m}= y_{N_{m+1}-1}$, $m\in \N$, we have
\begin{align*}
 \frac{1}{N_{m}} &
 \card  \{ 0\le i < N_{m}: \|(I+T)^{i}u_{m}\|
 \ge m\|u_{m}\|=m\} \\
 &
 \ge
 \frac{(m+1)!- m! -2 - (3( m-1)([((m+1)!)^\frac{2}{3}]+1)}
 {(m+1)!+1}
 \underset{m\to\infty}{\longrightarrow} 1.
\end{align*}
Moreover, $ I+T $ satisfies the Kitai Criterion (therefore, it is
hypercyclic) \cite[Theorem 1.1 and Corollary 1.3]{S08}, thus there
exists a dense sequence $(x_{k})_{k}$ in $X$ such that
$(I+T)^{n}x_{k} \rightarrow 0$ as $ n\rightarrow \infty $.

Hence the operator $I+T$ satisfies the conditions of  Corollary
\ref{Cor2} and    Theorem \ref{DDC}, and we conclude that $I+T$ has
a dense distributionally irregular manifold.
\end{proof}

\section{Acknowledgements }

The second author is supported in part by MEC and FEDER, Project MTM2008-05891.
The third and fourth authors are supported in part by MEC and FEDER, Projects
MTM2007-64222, MTM2010-14909 and by Generalitat Velenciana, Project PROMETEO/2008/101.



\end{document}